\documentclass[leqno,12pt]{amsart} 
\usepackage{amssymb} 
\theoremstyle{plain} 

\theoremstyle{definition} 

\usepackage{graphicx}
\usepackage[all,cmtip,line]{xy}
\usepackage{amssymb}

\usepackage{amsmath}
\usepackage{amsthm}
\usepackage{epstopdf}
\newtheorem{defin}{Definition}

\newtheorem{thm}{Theorem}

\newtheorem*{thmRH}{The Riemann-Hilbert Correspondence for Stacks}
\newtheorem*{theoremrh}{The Classical Riemann-Hilbert Correspondence}

\newcommand{\xs}{\mathcal{X}}
\newcommand{\xa}{\mathfrak{X}}

\newcommand{\C}{\mathbb{C}}

\newcommand{\adj}[4]{#1\negmedspace: #2\rightleftarrows #3:\negmedspace #4}

\newcommand{\shx}{\textbf{Sh}(\xa, \C)}
\newcommand{\shacx}{\textbf{Sh}_{c}(\xa, \C)}
\newcommand{\shcx}{\textbf{Sh}_{alc}(\xa, \C)}

\newcommand{\sis}{\textbf{Set}_{\Delta}}
\newcommand{\SC}{\mathcal{S}}
\newcommand{\dgt}{\textbf{dg-Cat}^{tri}}
\newcommand{\dgtd}{\textbf{dg-Cat}^{tri}_{\Delta}}
\newcommand{\dgti}{\textbf{dg-Cat}^{tri}_{\infty}}

\begin{document}

\title[The Riemann-Hilbert Correspondence for Algebraic Stacks]{The Riemann-Hilbert Correspondence for Algebraic Stacks} 

\author[Alexander G. M. Paulin]{Alexander G.M. Paulin} 

\address{Alexander G.M. Paulin
Department of Mathemtics\endgraf
King's College London \endgraf
Strand \endgraf
London, WC2R 2LS\endgraf
UK
}
\email{alexander.paulin@kcl.ac.uk}

\begin{abstract}
Using the theory $\infty$-categories we construct derived (dg-)categories of regular, holonomic $D$-modules and algebraically constructible sheaves on a complex smooth algebraic stack.  We construct a natural $\infty$-categorical equivalence between these two categories generalising the classical Riemann-Hilbert correspondence.
\end{abstract}

\maketitle

\section{Introduction} 
\noindent
Let $X$ be a smooth algebraic variety over $\C$ and $X^{an}$ be the complex analytic manifold associated to $X$.  Let $\mathcal{D}_{rh}^b(D(X))$ and $\mathcal{D}_c^b(X^{an}, \C)$ denote the derived categories of regular holonomic $D$-modules on $X$ and algebraically constructible complex sheaves on $X^{an}$ respectively.  Both categories are naturally triangulated and we equip $\mathcal{D}_{rh}^b(D(X))$  with the standard $t$-structure and $\mathcal{D}_c^b(X^{an}, \C)$ with the (middle) \textit{perverse} $t$-structure.  The theory of $D$-modules gives rise to a natural triangulated functor
$$\mathfrak{DR}_X: \mathcal{D}_{rh}^b(D(X))\rightarrow \mathcal{D}_c^b(X^{an}, \C),$$
known as the \textit{de Rham} functor.  The Riemann-Hilbert correspondence states that $\mathfrak{DR}_X$ is a $t$-exact equivalence of triangulated categories.  In this paper we generalise this to smooth algebraic stacks over $\C$ which admit an algebraic variety as a smooth atlas.

The first challenge in generalising this result is finding the appropriate concepts of derived categories of $D$-modules and constructible sheaves on smooth algebraic stacks.  The primary reason for this is that the category of triangulated categories is too crude for our purposes: we cannot \textit{glue} triangulated categories is a natural way.  In essence, this is because localising by quasi-isomorphisms discards too much information.  

Various enhancements to the classical theory of triangulated categories have been proposed to correct such defects.  Perhaps the most straightforward of these is the theory of triangulated differential graded (dg-)categories (\cite{TDG}).  A dg-category $\mathcal{C}$ is a category enriched over the category of complexes of $\C$-vector spaces.  To such a category we may naturally associate its homotopy category, denoted $h(\mathcal{C})$.  A dg-category $\mathcal{C}$ is triangulated if, roughly speaking, $h(\mathcal{C})$ is triangulated.  Thus triangulated categories provide an \textit{enhancement} of the category of triangulated categories.  

Another proposed enhancement is the theory of stable $\infty$-categories, as developed by Lurie (\S 1 \cite{LA}).  An $\infty$-category is, very roughly speaking, a higher category with the property that for $n >1$, all $n$-morphisms are invertible.  As in the case of dg-categories, an $\infty$-category $\mathcal{C}$ has a homotopy category $h(\mathcal{C})$.  If  $\mathcal{C}$ is stable  and then $h(\mathcal{C})$ is a triangulated category.  

There is a close relationship between these two approaches.  Given any triangulated dg-category $\mathcal{C}$ we may take its differential graded nerve, $N_{dg}(\mathcal{C})$, to get a $\C$-linear stable $\infty$-category.  In fact this construction gives an equivalence (in an $\infty$-categorical sense) between both theories.  We remark that this is only true because we have fixed the ground field $\C$; in positive characteristic they are not equivalent.  

The collection of all (small) triangulated dg-categories can naturally be arranged into an $\infty$-category, denoted $\dgti$.  This $\infty$-category admits limits (in an $\infty$-categorical sense), allowing us to suitably \textit{glue} triangulated categories.  This will be at the heart of our constructions. 
This approach closely follows that of Gaitsgory in his development of the categorical geometric Langlands correspondence.
\\ \\
\noindent
We now describe in detail the contents of this paper. 

In \S 2 we review the theory of $\infty$-categories and dg-categories, giving a detailed construction of $\dgti$.  

In \S 3 we review the theory of $D$-modules on smooth algebraic varieties, making suitable dg-enhancements of various classical triangulated categories.    Using this we construct the derived, triangulated dg-category of regular, holonomic $D$-modules on $\xs$, a smooth algebraic stack over $\C$, denoted $D_{rh}^b(\xs)$.  The essence of our construction is to define  $D_{rh}^b(\xs)$ as the $\infty$-categorical limit in $\dgti$ of classical dg-categories of $D$-modules over a suitable simplicial Cech cover of $\xs$.
This category is equipped with a standard $t$-structure coming from the standard $t$-structure of classical derived categories of $D$-modules.

In \S 4 we review the theory of construcible sheaves on complex analytic spaces.  As for the theory of $D$-modules, we construct a derived, triangulated dg-category of algebraically constructible sheaves on $\xs^{an}$, denoted  $\textbf{dg-Mod}_c^b(\xs^{an}, \C)$.
This category comes equipped with a \textit{perverse} $t$-structure, coming from the classical (middle) perverse $t$-structure on derived categories of construcible sheaves on complex analytic spaces.

In \S 5 we review the classical Riemann-Hilbert correspondence in the dg-setting.  Using the classical de Rham functor we construct a morphism
$$\widehat{\mathfrak{DR}_{\xs}}_{\infty}:D_{rh}^b(\xs)\rightarrow \textbf{dg-Mod}^b_c(\xs^{an}, \C)$$
in $\dgti$.
Our main theorem is the following: 
\begin{thmRH}
Let $\xs$ be a smooth complex algebraic stack which admits an algebraic variety as a smooth atlas.  Then the $\infty$-categorical de Rham functor $\widehat{\mathfrak{DR}_{\xs}}_{\infty}$ is an equivalence in $\dgti$.  Moreover it is $t$-exact and thus induces a canonical equivalence between the category of regular, holonomic $D$-modules on $\xs$ and the category of perverse sheaves on $\xs$.
\end{thmRH}

I would like to thank Dennis Gaitsgory for various useful conversations in the writing of this paper.

\section{Simplicial Sets and simplicial Categories}
In this paper we develop the theory of the Riemann-Hilbert correspondence on algebraic stacks using Joyal and Lurie's theory of $\infty$-categories. The main reference is the  foundational treatise \cite{LT}.  Following the terminology of \cite{LT},  by an $\infty$-category we mean an $(\infty, 1)$-category.  Loosely speaking this is a higher category such that for $n>1$,  all $n$-morphisms are invertible.  For the convenience of the reader we will review the aspects of the theory relevant to this paper.
\\ \\
\noindent
Let $\textbf{Cat}$ be the category of (small) categories. By convention, morphisms in $\mathbf{Cat}$ are given by covariant functors.  For $n \in \mathbb{Z}$, a non-negative integer, we define $[n]\in \textbf{Cat}$ to be the category with objects $\{0, \cdots, n \}$  and morphisms:
$$Hom_{[n]}(a,b) = 
\begin{cases} 
\emptyset  & \mbox{if } a>b \\
* & \mbox{if } a\leq b
\end{cases}$$
where $*$ denotes a unique morphism. We define the simplex category, denoted $\Delta$, to be the full subcategory of $\textbf{Cat}$  with objects $[n]$, for $n$ a non-negative integer.  Note that this category is a skeleton for the category whose objects are non-empty, finite, totally ordered sets and whose morphisms are non-decreasing functions between them.  
\begin{defin}
Let $\textbf{Set}$ denote the category of sets. A simplicial set is a functor
$$K: \Delta^{op} \rightarrow \textbf{Set}.$$
A morphism between simplicial sets is a natural transformation of the underlying functors. We denote the category of simplicial sets by $\sis$.
\end{defin}
\noindent
For $n\in \mathbb{Z}$, a non-negative integer, and $K\in \sis$, we define the set of $n$-cells of $K$ to be $K_n :=K([n])$. We call the $0$-cells the \textit{vertices} of $K$ and the $1$-cells the \textit{edges} of $K$.  We define the $n$-simplex to be the simplicial set $\Delta^n:= Hom_{\textbf{Cat}}( - , [n])$.  The vertices and edges of $\Delta^n$ are just the objects and morphisms of $[n]$ respectively. We say that an $m$-cell of $\Delta^n$ is non-degenerated if it corresponds to a monomorphism $[m]\rightarrow [n]$.  By the Yoneda lemma there is a natural bijection $K_{n} \cong Hom_{\sis}(\Delta^n, K)$. The boundary of $\Delta^n$, denoted $\partial\Delta^n$, is the simplicial set generated by $\Delta^n$ minus its unique non-degenerate $n$-cell.  For $0\leq m\leq n$ we define the $m^{th}$ horn $\Lambda^n_m\subset \Delta^n$ to be the simplicial set generated by $\partial\Delta^n$ minus the unique non-degenerate (n-1)-cell not containing $m$.  We call $\Lambda^n_m$ an \textit{inner} horn if $m\in \{1, \cdots, n-1\}$.
\\ \\
\noindent
Let \textbf{Top} denote the category of compactly generated, Hausdorff topological spaces.   As outlined in \S 1... of \cite{LT}, there is a natural adjunction
$$\adj{|\;|}{\sis}{\textbf{Top}}{Sing},$$
where $|.|$ is the \textit{geometric realisation} functor and $Sing$ is the \textit{singular complex} functor.  The geometric realisation of $\Delta^n$ in the usual topological $n$-simplex. 
Using this adjunction we define the Quillen model structure (\S A... \cite{LT}) on $\sis$ as follows:
\begin{enumerate}
\item A morphism of simplicial sets $f: X\rightarrow Y$ is a weak equivalence if the morphism $|f|: |X|\rightarrow |Y|$ is a weak homotopy equivalence.
\item The fibrations are the Kan fibrations, i.e. those maps which have the right-lifting property with respect to all horn inclusions  $\Lambda^n_m\subset \Delta^n$
\item The cofibrations are the monomorphisms. 
\end{enumerate}
This gives $\sis$ the structure of a combinatorial model category such that every simplicial set is cofibrant.  
\begin{defin}
A simplicial set which is fibrant with respect to the Quillen model structure is called an $\infty$-groupoid (or Kan complex).  More precisely, a simplicial set $K$ is an $\infty$-groupoid if and only if it satisfies the following property:
\noindent
Given $m\in \{0, \cdots, n\} $, any morphism $\Lambda_m^n\rightarrow K$, admits an extension to a morphism $\Delta^n \rightarrow K$.
\noindent
\end{defin}
\noindent
If $X \in \mathbf{Top}$ then $Sing(X)$ is an $\infty$-groupoid.  
\\ \\
\noindent
Giving \textbf{Top} its classical model structure (fibrations are Serre fibrations), the above adjunction becomes a Quillen equivalence. We define the homotopy category of spaces, denoted $\mathcal{H}$, as the homotopy category of $\sis$ with respect to the Quillen model structure. By the above, this is canonically equivalent to the classical homotopy category of spaces. 
\\ \\
\noindent
The cartesian product gives $\textbf{Set}$ the structure of a symmetric monoidal category.  This induces a symmetric monoidal structure on $\sis$ in the obvious way.  

\begin{defin}For $K, L \in \sis$ we define the simplicial set of maps from $K$ to $L$, denoted $Map(K,L)\in \sis$ as follows:
$$ Map(K,L)_n := Hom_{\sis} (K\times \Delta^n, L).$$
\end{defin}
\noindent
This gives $\sis$ the structure of a closed symmetric monoidal simplicial model category (\S A.... \cite{LT}).  This in turn gives $\mathcal{H}$ the structure of a symmetric monoidal category such that the localisation functor $\sis\rightarrow \mathcal{H}$ is symmetric monoidal.
\\ \\
\noindent
The category $\mathbf{Set}$ provides the foundation for classical category theory. More precisely, the definition of a category relies on both $\mathbf{Set}$ and its natural symmetric monoidal structure coming from the cartesian product.  If $\mathcal{S}$ is any symmetric monoidal category, we can replace $\mathbf{Set}$ with $\mathcal{S}$ in the definition to give the theory of $\mathcal{S}$-enriched categories.  For a detailed discussion of enriched category theory we refer the reader to \S A.1.4 of \cite{LT}.   Roughly speaking, an $\mathcal{S}$-enriched category $\mathcal{C}$ is a class of objects such that for any two objects $a, b\in \mathcal{C}$ there is a mapping object $Map_{\mathcal{C}}(a,b)\in \mathcal{S}$, with the usual extra structure. We reserve the term \textit{hom} exclusively for the classical case. For this perspective, an ordinary category is just a $\mathbf{Set}$-enriched category.  Many categories we naturally encounter are in some way enriched.  For example, if $k$ is a field, categories enriched over $k$-vector spaces are called $k$-linear.  The category of $k$-vector spaces is itself $k$-linear.  
  
One approach to higher category theory is to replace $\mathbf{Set}$ with a suitable symmetric monoidal model category $\mathcal{S}$ (\S 1.1\cite{LT}). The most natural category to consider is $\mathbf{Top}$.  In this case, the concept of a $2$-morphism would be a path between  $1$-morphisms, a $3$-morphism a homotopy between paths, and so on.  This is a perfectly valid approach but given the fact that $\mathbf{Top}$ and $\sis$ are Quillen equivalence model categories we are free to use the latter category.

\begin{defin}
A simplicial category is a category which is enriched over the category $\sis$ of simplicial sets (with respect to the natural symmetric monoidal structure). The category of (small) simplicial categories (where morphisms are given by simplicially enriched functors) will be denoted by $\textbf{Cat}_{\Delta}$.
\end{defin}
\noindent
 \noindent
For a general symmetric monoidal category $\mathcal{S}$ we denote by $\mathbf{Cat}_{\mathcal{S}}$, the category of (small) $\mathcal{S}$-enriched categories. If $\mathcal{T}$ is a second symmetric monoidal category and $f:\mathcal{S}\rightarrow \mathcal{T}$ is a symmetric monoidal functor then $f$ induces a functor $\mathbf{Cat}_{\mathcal{S}}\rightarrow \mathbf{Cat}_{\mathcal{T}}$.
The constant functor $\textbf{Set}\rightarrow \sis$ is symmetric monoidal, hence we may regard an ordinary category as a simplicial category by identifying hom-sets with their constant simplicial sets.  

\begin{defin}
For $\mathcal{S}$, a symmetric monoidal category, and $\mathcal{C}$ an $\mathcal{S}$-enriched category, the (ordinary) category underlying $\mathcal{C}$, denoted $\mathcal{C}_0$, is defined as follows:

\begin{enumerate}
\item The objects of $\mathcal{C}_0$ are the same as the objects of $\mathcal{C}$.
\item For $a, b \in \mathcal{C}_{0}$, $Hom_{\mathcal{C}_{0}}(a,b):= Hom_{\mathcal{S}}(1_{\mathcal{S}}, Map_{\mathcal{C}}(a,b))$,
\end{enumerate}
where $Map_{\mathcal{C}}(a,b)\in \mathcal{S}$ is the mapping object from $a$ to $b$ and $1_{\mathcal{S}}$ is the unit object in $\mathcal{S}$.
\end{defin}
\noindent
In the case when $\mathcal{S}=\sis$, and $\mathcal{C}$ is a simplicial category, the hom-sets in $\mathcal{C}_0$ are the 0-cells of the mapping simplicial sets.  The symmetric monoidal functor $\sis\rightarrow \mathcal{H}$ allows us to consider $\mathcal{C}$ as an $\mathcal{H}$-enriched category, denoted $\tilde{h}(\mathcal{C})$.  We define the homotopy category of $\mathcal{C}$, denoted $h(\mathcal{C})$, to be the
category underlying $\tilde{h}(\mathcal{C})$. There is a canonical functor from $\mathcal{C}_0$ to $h(\mathcal{C})$. The formation of the homotopy category is functorial.  

A morphism $f:\mathcal{C}\rightarrow \mathcal{D}$, between two simplicial categories is called a \textit{weak equivalence} if $\tilde{h}(f):\tilde{h}(\mathcal{C})\rightarrow \tilde{h}(\mathcal{D})$ is an equivalence of $\mathcal{H}$-enriched categories (\S A 3.2.1 \cite{LT}). The category $\mathbf{Cat}_{\Delta}$ admits a natural model structure (called the Bergner model structure, see \S A 3.2.4 of \cite{LT}) with the above weak equivalences. 

The principal weakness of simplicial categories as a model for higher category theory is that the correct notion of functor (in a higher sense) should be a \textit{homotopy coherent} diagram, a more general notion that a simplicially enriched functor.  Roughly speaking, a homotopy coherent diagram is a weakened functor where the associativity conditions only hold up to specified collection of higher homotopies.  In the next section we introduce an alternate, but closely related, model for higher category theory where this issue is neatly resolved.
\section*{The Nerve Functor and $\infty$-categories}
\noindent
For a detailed treatment of the material in this section we refer the reader to \S1.1 of \cite{LT}. 
\\ \\
\noindent
Let $\mathcal{C}$ be an ordinary category.  The \textit{nerve} of $\mathcal{C}$, denoted $N(\mathcal{C})$, is the simplicial set defined as follows: for $n\in \mathbb{Z}$, a non-negative integer,  $N(\mathcal{C})_ n := Hom_{\textbf{Cat}}([n], \mathcal{C})$.  More concretely $N(\mathcal{C})_ n$ is the set of all composable strings of morphisms:
$$C_0\rightarrow \cdots \rightarrow C_n.$$
\noindent
Thus the $0$-cells of $N(\mathcal{C})$ may be identified with the objects of $\mathcal{C}$ and $1$-cells with morphisms of $\mathcal{C}$.  The reader cautious of set theoretic issues is referred to \S 1.1.15 of \cite{LT}. As explained in \S1.1.2 of \cite{LT}, we can canonically recover $\mathcal{C}$ from $N(\mathcal{C})$.  Moreover the nerve defines a fully-faithful functor from $\textbf{Cat}$ to $\sis$. 
By \S1.2.2.2 of \cite{LT}, if $\mathcal{C}$ is an ordinary category, then $N(\mathcal{C})$ has the following important property:
\\ \\
\noindent
For $n$ a positive integer greater than 1 and  $m\in \{1, \cdots, n-1\} $, any morphism $\Lambda_m^n\rightarrow N(\mathcal{C})$ admits a \textit{unique} extension to a morphism $\Delta^n \rightarrow N(\mathcal{C})$.
\\ \\
\noindent
This gives a complete description of the essential image of the nerve functor.  Notice that this extension property is only guaranteed to hold for the \textit{inner} horns; the $\infty$-groupoid condition involved all horns but drops the uniqueness. These two examples motivate the following fundamental definition:
\begin{defin}
An $\infty$-category is a simplicial set $K$ which satisfies the following property: For $n$ a positive integer greater than 1 and  $m\in \{1, \cdots, n-1\} $, any morphism $\Lambda_m^n\rightarrow K$ admits a \textit{not necessarily unique} extension to a morphism $\Delta^n \rightarrow K$.
\end{defin}
\noindent
It is immediately clear that the nerve of an ordinary category is an $\infty$-category.  Similarly an $\infty$-groupoid is an $\infty$-category. Thus the theory of $\infty$-categories simultaneously generalises (ordinary) category theory and topology. 
\\ \\ 
\noindent
The nerve functor admits a natural extension to all simplicial categories as explained in  \S 1.1.5 of \cite{LT}.   This new functor, again denoted by $N$, is sometimes called the simplicial (or homotopy coherent) nerve.  It is part of an adjunction: 
$$\adj{\mathfrak{C}}{\sis}{\mathbf{Cat}_{\Delta}}{N}.$$
As explained in \S 2.2.5.1 of \cite{LT}, there is an alternate model structure on $\sis$ (called the Joyal model structure) where the weak equivalence are defined as follows:

\begin{defin}
Let $S, T \in \sis$ and $f:S\rightarrow T$ be a morphism.  We say that $f$ is a \textit{categorical equivalence} if the induced functor $\mathfrak{C}(f): \mathfrak{C}(S)\rightarrow \mathfrak{C}(T)$, is a weak equivalence with respect to the Bergner model structure on $\mathbf{Cat}_{\Delta}.$
\end{defin}
\noindent
Putting the Joyal model structure on $\sis$, the above adjunction becomes a Quillen equivalence.  The fibrant-cofibrant objects with respect to the Joyal model structure on $\sis$ are precisely the $\infty$-categories. We say that two  $\infty$-categories are \textit{equivalent} if they are categorically equivalent as simplicial sets.  

It is not true that the nerve of any simplicial category is an $\infty$-category.  If, however, the mapping spaces between all objects in a simplicial category are  $\infty$-groupoids then its nerve is an $\infty$-category.  We call any simplicial category with this property fibrant.  Thus we can in some senses think about an $\infty$-category as a category enriched over $\infty$-groupoids.  
\\ \\
\noindent
Let $K$ be an $\infty$-category. The \textit{objects} of $K$ are defined to be the vertices  of the underlying simplicial set.  Thus an object in $K$ is a map of simplicial sets $\Delta^0\rightarrow K$.  We write $a\in K$ to denote an object.  Similarly, \textit{morphisms} of $K$ are defined to be the edges  of the underlying simplicial set.  More precisely a morphism is a map of simplicial sets $f: \Delta^1 \rightarrow K$.  The simplex $\Delta^1$ has two vertices $\{0\}$ and $\{ 1\}$.  Thus any morphism has a source object $f(\{0\}) = a$ and target object $f(\{1\}) = b$.  In the usual way, we express this information as $f:a\rightarrow b$.  For $a\in K$ we define identity morphisms $id_a:a\rightarrow a$ to be the unique extension of $a$ to an edge. 

The inner horn condition guarantees that in an $\infty$-category there is a way to compose two morphisms with the same source and target.  Note however, that a choice of composition is only unique up to a contractible space.  This is perhaps the most conceptually challenging aspect of working with $\infty$-categories as a model for higher category theory.
\\ \\
\noindent
Let $K$ be an $\infty$-category and $a,b \in K$.  We define the space of maps from $a$ to $b$ to be the simplicial set $Map_{K}^R(a,b)$, whose $n$-cells are those morphisms $z: \Delta^{n+1}\rightarrow \mathcal{C}$, such that $z|\Delta^{\{n+1\}} = b$ and $z|\Delta^{\{0, \cdots, n\}}$ is the constant $n$-cell at the vertex $a$.  By \S1.2.2.3 of \cite{LT}, this simplicial set is an $\infty$-groupoid.  It can be shown that if $\mathcal{C}$ is a fibrant simplicial category and $a, b \in \mathcal{C}$, then the $Map_{\mathcal{C}}(a,b)$ is weakly equivalent (for the Quillen model structure on $\sis$) to  $Map_{N(\mathcal{C})}^R(a,b)$.   An object $b \in K$ is said to be $\textit{final}$ if for any $a\in K$ the $\infty$-groupoid $Map_{K}^R(a,b)$ is weakly contractible.
\\ \\
\noindent
An  $\infty$-\textit{functor} between two $\infty$-categories is a natural transformation of the underlying simplicial sets.  This is one of the principal reasons for using $\infty$-categories: we do not need to introduce coherent homotopy diagrams as they are encoded by the underlying simplicial set of an $\infty$-category.  
\\ \\
\noindent
As explained in \S 1.2.3 of \cite{LT}, the nerve functor $N: \textbf{Cat}\rightarrow \sis$ admits a left adjoint $h: \sis\rightarrow\textbf{Cat}$.  If $K\in \sis$ then $h(K)$ is called the \textit{homotopy} category  of $K$. If $\mathcal{C}$ is a fibrant simplicial category then there is a natural equivalence $h(N(\mathcal{C}))\cong h(\mathcal{C})$, where $h(\mathcal{C})$ is the homotopy category introduced in the previous section.  Similarly, if $K$ is an $\infty$-category then there is an equivalence $h(\mathfrak{C}(K))\cong h(K)$.  In the case when $K$ is an $\infty$-category, $h(K)$ admits a more concrete description:  
\begin{enumerate}
\item The objects of $h(K)$ are the objects of $K$.
\item For $X, Y\in h(K)$ the set of morphism from $X$ to $Y$ is the set of \textit{homotopy} classes of morphisms (in $K$) $f: X\rightarrow Y$, denoted $[f]$.
\end{enumerate}
\noindent
Two morphisms $f,g: X\rightarrow Y$ are said to be homotopic, if there exists a 2-cell in $K$ whose boundary is given by:
$$
\xymatrix{
&Y  \ar[dr]^{id_Y}\\
X\ar[ur]^{f} \ar[rr]^{g} & &Y
}$$
As proven in \S1.2.3 of \cite{LT}, this is an equivalence relation when $K$ is an $\infty$-category.  If we have two morphisms $f:X\rightarrow Y$ and $g: Y\rightarrow Z$ in $K$ then these defines a morphism $\Lambda_1^2\rightarrow K$, which we represent by the diagram:
$$
\xymatrix{
&Y  \ar[dr]^{g}\\
X\ar[ur]^{f} & &Z
}$$
By the defining property of $\infty$-categories, we may extend this to a 2-simplex $\Delta^2\rightarrow K$.  We may then take its boundary:
$$
\xymatrix{
&Y  \ar[dr]^{g}\\
X\ar[ur]^{f} \ar[rr]^{g\circ f}& &Z
}$$
Note that the morphism $g\circ f$ is not necessarily unique determined by $f$ and $g$:  it depends on the choice of $2$-simplex.  What is true though, is that it is unique up to homotopy.  Thus we define composition in $h(K)$ to be 
$$[g]\circ [f]:= [g\circ f].$$
This is independent of all choices.  If $\mathcal{C}$ is an ordinary category then $h(N(\mathcal{C}))$ is canonically isomorphic to $\mathcal{C}$.   It can be shown that an $\infty$-category $K$ is an $\infty$-groupoid if and only if $h(K)$ is a groupoid in the usual sense.  

If $K$ is an $\infty$-category (or a simplicial category), we say a morphism, $f:X\rightarrow Y$ in $K$, is an \textit{equivalence} if it becomes an isomorphism in $h(K)$.   Thus we see that an $\infty$-groupoid is an $\infty$-category in which every morphism is an equivalence.  

In \S1.2.5 of \cite{LT}, it is shown that given $K$, an $\infty$-category, there exists a \textit{largest} $\infty$-groupoid $K'\subset K$.  Moreover the functor $K\rightarrow K'$ from $\infty$-categories to $\infty$-groupoids is right adjoint to the natural inclusion of $\infty$-groupoids in $\infty$-categories.
\section*{Differential Graded Categories}
\noindent
For a detailed treatment of the material in this section we refer the reader to \cite{TDG} and \S1.3.1 of \cite{LA}.
\\ \\
\noindent
The symmetric monoidal (Quillen) model category $\sis$ provides the foundation for the theory of simplicial categories.  We now introduce another important, and closely related, class of enriched categories which give a good model for higher category theory.
\\ \\
\noindent
Let $k$ be a field and let $Ch(k)$ denote the category of chain complexes of $k$-vector spaces.  Recall that $Ch(k)$ has a natural closed monoidal model structure, where the product is given by the usual tensor product of chain complexes and the model structure is defined as follows:

\begin{enumerate}
\item Weak equivalences are quasi-isomorphisms.
\item Fibrations are epimorphisms.
\item Cofibrations are monomorphisms.
\end{enumerate}
The homotopy category of $Ch(k)$ is the derived category of $k$, denoted by $D(k)$.  Because $k$ is a field, $D(k)$ is equivalent to the category of $\mathbb{Z}$-graded $k$-vector spaces.  There is a symmetric monoidal structure on $D(k)$, defined in the obvious way, making the localisation functor $Ch(k)\rightarrow D(k)$ symmetric monoidal.

\begin{defin}
A differential graded category (dg-category for short) over $k$, is a category enriched over $Ch(k)$.   The collection of all (small) dg-categories may be arranged into a category whose objects are (small) dg-categories and whose morphisms are $Ch(k)$-enriched functors. As above, we denote this category by $\mathbf{Cat}_{Ch(k)}$.
\end{defin}
\noindent
For the rest of this paper, by a dg-category we mean a dg-category over $k$. Let $\mathcal{C}$ be a dg-category.   For $X,Y \in \mathcal{C}$ we denote the mapping complex by $Map_{\mathcal{C}}(X,Y)$:
$$\xymatrix{\cdots \ar[r] &Map_{\mathcal{C}}(X,Y)_{-1}\ar[r]^d &Map_{\mathcal{C}}(X,Y)_0\ar[r]^d &Map_{\mathcal{C}}(X,Y)_1 \ar[r]&\cdots
}$$
\noindent
The category underlying $\mathcal{C}$ (as a $Ch(k)$-enriched category) has hom-sets given by:
 $$Hom_{\mathcal{C}_0}(X,Y):= Hom_{Ch(k)}(1_{Ch(k)},Map_{\mathcal{C}}(X,Y)) =\{f\in Map_{\mathcal{C}}(X,Y)_0 | df=0\}.$$ This makes it clear that the category underlying a dg-category is $k$-linear.

Because $Ch(k)$ is a symmetric monoidal model category, any dg-category $\mathcal{C}$ naturally gives rise to a $D(k)$-enriched category $\tilde{h}(\mathcal{C})$.  As for simplicial category theory we have the following definition:

\begin{defin}
Let $\mathcal{C}$ be a dg-category.  The homotopy category of $\mathcal{C}$, denoted $h(\mathcal{C})$, is the category underlying $\tilde{h}(\mathcal{C})$.
\end{defin}
\noindent
Concretely, for $X,Y\in h(\mathcal{C})$, we have a natural bijection $Hom_{h(\mathcal{C})}(X,Y)\cong H^0(Map_{\mathcal{C}}(X,Y))$. Note that this implies that the homotopy category of a dg-category is canonically $k$-linear. As in the simplicial case, the formation of the homotopy category is functorial.
\\ \\
\noindent
There is a natural model structure on $\mathbf{Cat}_{Ch(k)}$ (\S A.3.2.4 \cite{LT}) which has the following class of weak equivalences:
\begin{defin}
Let $\mathcal{C}, \mathcal{D} \in \textbf{Cat}_{Ch(k)}$.   We say that a morphism $f: \mathcal{C}\rightarrow \mathcal{D}$ is a \textit{weak equivalence} if $\tilde{h}(f)$ is an equivalence of $D(k)$-enriched categories.  More precisely, when the following conditions are satisfied:
\begin{enumerate}
\item For any two objects $X,Y \in \mathcal{C}$ the morphism $$f_{XY}: Map_{\mathcal{C}}(X,Y)\rightarrow Map_{\mathcal{D}}(f(X), f(Y))$$ is a quasi-isomorphism of chain complexes over $k$.
\item The induces functor $h(f): h(\mathcal{C})\rightarrow h(\mathcal{D})$ is an equivalence of categories.
\end{enumerate}
\end{defin}
\noindent 
As explained in  \S 1.3.1 of \cite{LA}, the Dold-Kan correspondence allows us to transform a dg-category into a fibrant simplicial category.  Applying the (simplicial) nerve functor we can further transform a dg-category into an $\infty$-category. This process is simplified in  \S 1.3.1.6 of \cite{LA}, where Lurie directly constructs a differential graded nerve functor: 
$$N_{dg}: \textbf{dg-Cat}\rightarrow \sis.$$
 If $\mathcal{C}$ is a dg-category then it is straightforward to describe the low dimensional cells of $N_{dg}(\mathcal{C})$:

\begin{enumerate}
\item The $0$-cells of $N_{dg}(\mathcal{C})$ are the objects of $\mathcal{C}$.
\item The $1$-cells of $N_{dg}(\mathcal{C})$ are the morphisms in the underlying category of $\mathcal{C}$.
\item The $2$-cells of $N_{dg}(\mathcal{C})$ are given by the following data:  objects $X,Y, Z \in \mathcal{C}$; morphisms $f \in Hom_{\mathcal{C}_0}(X,Y), g\in Hom_{\mathcal{C}_0}(Y,Z), h\in Hom_{\mathcal{C}_0}(X,Z)$ and an element $z\in Map_{\mathcal{C}}(X,Z)_{-1}$ such that $dz = h - (g\circ f).$
\end{enumerate} 
\noindent
As proven in \S1.3.1.17 of \cite{LA}, given a dg-category $\mathcal{C}$, the differential graded nerve $N_{dg}(\mathcal{C})$ is an $\infty$-category which is categorically equivalent to the $\infty$-category given by the construction utilising the Dold-Kan correspondence and the simplicial nerve functor.  There is also a canonical equivalence of homotopy categories:
$$h(\mathcal{C}) \cong h(N_{dg}(\mathcal{C})).$$
\noindent
Just as for simplicial categories, the correct notion of (higher) functors between dg-categories should be given by homotopy coherent diagrams, appropriately generalising  $Ch(k)$-enriched functors. Thankfully, the differential graded nerve gives us a straightforward way to make this precise.  

\begin{defin} An $\infty$-functor between two differential graded categories $\mathcal{C}$ and $\mathcal{D}$ is a $k$-linear $\infty$-functor of the underlying $\infty$-categories, $N_{dg}(\mathcal{C})$ and $N_{dg}(\mathcal{D})$.  An $\infty$-functor $f: N_{dg}(\mathcal{C})\rightarrow N_{dg}(\mathcal{D})$ is $k$-linear if the induced functor $h(f):h(\mathcal{C})\rightarrow h(\mathcal{D})$ is enriched over $k$-vector spaces.
\end{defin}
\noindent
Using this, we may naturally arrange (small) dg-categories into an $\infty$-category as follows:
\begin{defin}
Let $\textbf{dg-Cat}_{\Delta}$ be the fibrant simplicial category defined as follows:
\begin{enumerate}
\item The objects of $\textbf{dg-Cat}_{\Delta}$ are (small) dg-categories over $k$.
\item Given $\mathcal{C}, \mathcal{D}\in \textbf{dg-Cat}_{\Delta}$, the mapping space $Map_{\textbf{dg-Cat}_{\Delta}}(\mathcal{C}, \mathcal{D})$ is the largest $\infty$-groupoid contained in the restriction of $Map_{\sis}(N_{dg}(\mathcal{C}), N_{dg}(\mathcal{D}))$ to $k$-linear $\infty$-functors.
\end{enumerate}
We define the $\infty$-category of (small) dg-categories to be $\textbf{dg-Cat}_{\infty}= N(\textbf{dg-Cat}_{\Delta})$.  We denote the category underlying $\textbf{dg-Cat}_{\Delta}$ by $\textbf{dg-Cat}$.
\end{defin}

Following \S 4.4 of \cite{TDG} there is a subclass of dg-categories called triangulated.  We refer the reader to \S 4.4.7 for a precise definition. Being triangulated implies that the homotopy category is a triangulated category in the classical sense.  Thus the theory of triangulated dg-categories is an enhancement of the theory of triangulated categories.   

\begin{defin}Let $\mathcal{C}$ and $\mathcal{D}$ be triangulated dg-categories.  We say that an $\infty$-functor $f$ from $\mathcal{C}$ to $\mathcal{D}$ is exact if the induced (ordinary) functor $h(f):h(\mathcal{C})\rightarrow h(\mathcal{D})$ is an exact functor of triangulated categories.
\end{defin}
\begin{defin}
Let $\textbf{dg-Cat}^{tri}_{\Delta}$ be the fibrant simplicial category defined as follows:
\begin{enumerate}
\item The objects of $\textbf{dg-Cat}^{tri}_{\Delta}$ are (small) triangulated dg-categories over $k$.
\item Given $\mathcal{C}, \mathcal{D}\in \textbf{dg-Cat}^{tri}_{\Delta}$, the mapping space $Map_{\textbf{dg-Cat}_{\Delta}}(\mathcal{C}, \mathcal{D})$ is the largest $\infty$-groupoid contained in the restriction of $Map_{\sis}(N_{dg}(\mathcal{C}), N_{dg}(\mathcal{D}))$ to exact $k$-linear $\infty$-functors.
\end{enumerate}
We define the $\infty$-category of (small) triangulated dg-categories to be $\textbf{dg-Cat}^{tri}_{\infty}= N(\textbf{dg-Cat}^{tri}_{\Delta})$.  We denote the category underlying $\textbf{dg-Cat}^{tri}_{\Delta}$ by $\textbf{dg-Cat}^{tri}$.
\end{defin}
\noindent
The differential graded nerve functor preserves weak equivalences, with respect to the Joyal model structure on $\sis$.  Thus if $\mathcal{C}$ and $\mathcal{D}$ are dg-categories and $f:\mathcal{C}\rightarrow \mathcal{D}$ is a weak equivalence, the $\infty$-functor $N_{dg}(f)$ is an equivalence in the $\infty$-category $\textbf{dg-Cat}_{\infty}$.  
\\ \\
\noindent
By a $t$-structure on a triangulated dg-category $\mathcal{C}$ we we mean a $t$-structure on the triangulated category $h(\mathcal{C})$.  We define the heart  of a $t$-structure on $\mathcal{C}$, denoted $\mathcal{C}^{\heartsuit}$,  to be heart of the $t$-structure on the underlying homotopy category.  
Given an $\infty$-functor between triangulated dg-categories $f: \mathcal{C}\rightarrow \mathcal{D}$, each equipped with a $t$-structure, we say that $f$ is left $t$-exact (respectively right $t$-exact) if $h(f): h(\mathcal{C})\rightarrow h(\mathcal{D})$ is left $t$-exact (respectively right  $t$-exact) in the classical sense.  

The homotopy functor from from $\dgt$ to $\mathbf{Cat}$ induces a functor $h(\dgti)\rightarrow h(\mathbf{Cat})$, where $h(\mathbf{Cat})$ is the localisation of $\mathbf{Cat}$ by equivalences.  This latter category is equivalent to the category $[\mathbf{Cat}]$, whose objects are (small) categories and whose  morphisms are isomorphism classes of functors.  If we take a morphism in $h(\dgti)$ then it makes sense to talk about it being $t$-exact because being $t$-exact is an invariant of the isomorphism class of a functor.

\section*{Homotopy Limits and Limits in $\infty$-Categories}
\noindent
For a detailed treatment of the material in this section we refer the reader to \S1.2.13 of \cite{LT}.
\\ \\
\noindent
There is a natural notion of limits in $\infty$-categories generalising the classical case.  Let $\mathcal{C}$ and $\mathcal{D}$ be $\infty$-categories and $p: \mathcal{C}\rightarrow\mathcal{D}$ be an $\infty$-functor.  As explained in \S 1.2.9 of \cite{LT}, mimicking the classical construction, we may form \textit{the $\infty$-category of objects of $\mathcal{D}$ over $p$}, denoted $\mathcal{D}_{/p}$.  The $n$-cells of $\mathcal{D}_{/p}$ are the morphisms of simplicial sets
$$\Delta^n \star \mathcal{C}\rightarrow \mathcal{D},$$
whose restriction to $\mathcal{C}$ is $p$.  Here $\star$ denotes the \textit{join} of two $\infty$-categories, as constructed in \S1.2.8 of \cite{LT}. Note that because $\Delta^0 \star \mathcal{C}$, denoted $\mathcal{C}^{\triangleleft}$, has a privileged object \{0\} (the cone point), any object $X\in  \mathcal{D}_{/p}$ gives a canonical object $X(\{0\})\in \mathcal{D}$. 

\begin{defin}
A limit of the $\infty$-functor $p: \mathcal{C}\rightarrow\mathcal{D}$ is a final object in the $\infty$-category $\mathcal{D}_{/p}$. If a limit exists we denote it by $\varprojlim p$.
\end{defin}
\noindent
By definition limits are unique up to a weakly contractible space.  By an abuse of notation we will often refer to the limit of $p$ as the object $$lim(p):=\varprojlim p(\{0\})\in \mathcal{D}.$$
\noindent
Let us now relate $\infty$-categorical limits to more classical homotopy limits.  Let $\SC$ be a combinatorial model category.  For example, $\sis$ equipped with either the Quillen or Joyal model structures.  Let $\mathcal{I}$ be a small (ordinary) category.  By definition, $\SC$ admits small limits, thus the constant functor
$$\SC\rightarrow \SC^{\mathcal{I}}:=Fun(\mathcal{I}, \SC),$$
admits a right adjoint, denoted $lim:\SC^{\mathcal{I}}\rightarrow \SC$.  On objects this sends a functor $\mathcal{F}:\mathcal{I}\rightarrow\SC$ to $lim(\mathcal{F})$ in the usual sense.  The model structure on $\SC$ induces a natural model structure on the functor category $\SC^{\mathcal{I}}$, making $lim$ a right Quillen functor.  Thus we may form the right derived functor

$$Rlim: h(\SC^{\mathcal{I}})\rightarrow h(\SC),$$
by composing $lim$ with a fibrant replacement functor.  As usual,  $h$ denotes the homotopy category of the underlying model category.  

\begin{defin}If $\mathcal{F}\in\SC^{\mathcal{I}}$ then we define the homotopy limit of $\mathcal{F}$ to be  $holim(\mathcal{F}):= Rlim(\mathcal{F}),$
after identifying $\mathcal{F}$ with its image under the localisation functor $\SC^{\mathcal{I}}\rightarrow h(\SC^{\mathcal{I}})$.
\end{defin}
\noindent
From now on let $\SC$ be $\sis$ equipped with the Quillen model structure and let $\mathcal{C}$ be a simplicial category. As above, let $\mathcal{C}_0$ denote the ordinary category underlying $\mathcal{C}$, and let $\mathcal{I}$ denote an ordinary small category.  We denote by $\mathcal{I}_{\Delta}$, the simplicial category associated to $\mathcal{I}$.
\noindent
Let $\mathcal{F}:\mathcal{I}\rightarrow \mathcal{C}_0$ be a functor and choose $B\in \mathcal{C}_0$ together with a compatible collection of morphisms
$$\eta_{\mathcal{I}}:=\{B\rightarrow \mathcal{F}(i)\}_{i\in \mathcal{I}}.$$
We remark that this is equivalent to choosing a cone over $\mathcal{F}$ with vertex $B$.   For any $A\in \mathcal{C}_0$, this data induces a morphism of simplicial sets
$$Map_{\mathcal{C}}(A, B)\rightarrow lim(\mathcal{F}_A),$$
where  $\mathcal{F}_A: \mathcal{I}\rightarrow \sis$ is the functor sending $i\in \mathcal{I}$ to $Map_{\mathcal{C}}(A, \mathcal{F}(i)).$
This further induces a map of simplicial sets
$$Map_{\mathcal{C}}(A, B)\rightarrow holim(\mathcal{F}_A).$$
We say that $\eta$ \textit{exhibits} $B$ as a homotopy limit of $\mathcal{F}$ if each such morphism is a weak equivalence in $\sis$ for all $A\in \mathcal{C}_0$.
\\ \\
\noindent
We now relate this to $\infty$-categorical limits.  The functor $\mathcal{F}$ induces a simplicial functor:
$$\mathcal{F}_{\Delta}:\mathcal{I}_{\Delta}\rightarrow \mathcal{C}.$$
The nerve functor gives a morphism of simplicial sets 
$$\mathcal{F}_{\infty}= N(\mathcal{F}_{\Delta}):N(\mathcal{I}_{\Delta})\rightarrow N(\mathcal{C}).$$
Let us now make the additional assumption that $\mathcal{C}$ is a fibrant simplicial category.  Thus $N(\mathcal{C})$ is an $\infty$-category. The data of $\eta_{\mathcal{I}}$ induces a morphism
$$\overline{\mathcal{F}_{\infty}}:N(\mathcal{I}_{\Delta})^{\triangleleft}\rightarrow N(\mathcal{C}),$$
extending $\mathcal{F}_{\infty}$, with cone point $B$.  The following result of Lurie (\S4.2.4.1 \cite{LT}) is fundamental:
\begin{thm}  With the same notation as above, the following are equivalent
\begin{enumerate}
\item The data $\eta_{\mathcal{I}}$ exhibits $B$ as a homotopy limit of $\mathcal{F}$.
\item The functor $\overline{\mathcal{F}_{\infty}}$ is a limit diagram of $\mathcal{F}_{\infty}$.
\end{enumerate}
\end{thm}
\noindent
This shows that there is a close relationship between homotopy limits in fibrant simplicial categories and $\infty$-categorical limits.  
\begin{thm}
Let $\mathcal{C}$ be a fibrant simplicial category. Let $\mathcal{F}, \mathcal{G}\in Fun(\mathcal{I}, \mathcal{C}_0)$ be two ordinary functors.  Assume that $\varphi:\mathcal{F}\rightarrow \mathcal{G}$ is a natural transformation which becomes a natural isomorphism after composing both $\mathcal{F}$ and $\mathcal{G}$ with the canonical functor $\mathcal{C}_0\rightarrow h(\mathcal{C}).$  If $lim(\mathcal{F}_{\infty})$ and  $lim(\mathcal{G}_{\infty})$ exist then $\varphi$ induces a canonical (up to homotopy) equivalence $\varphi_{\infty}: lim(\mathcal{F}_{\infty})\rightarrow lim(\mathcal{G}_{\infty})$.
\end{thm}
\begin{proof}
We will show that if $lim(\mathcal{F}_{\infty})$ and $lim(\mathcal{G}_{\infty})$ exist, then $\varphi$ induces an canonical isomorphism (in $h(\mathcal{C})$) between them.  To do this we show that $\varphi$ induces a canonical isomorphism of their respective hom-functors in $h(\mathcal{C})$.

Let $A\in \mathcal{C}$. The natural transformation $\varphi$ induces a natural transformation $\varphi_A: \mathcal{F}_A\rightarrow \mathcal{G}_A$ in $\sis^{\mathcal{I}}$.  By \S 1.2.4.1 of \cite{LT} we know that this is a weak equivalence for the natural model structure on the diagram category $\sis^{\mathcal{I}}$.
Thus $\varphi_A$ induces a weak equivalence $holim(\mathcal{F}_A)\rightarrow holim(\mathcal{G}_A)$.  

As above, let $\varprojlim(\mathcal{F}_{\infty}), \varprojlim(\mathcal{G}_{\infty}): N(\mathcal{I}_{\Delta})^{\triangleleft}\rightarrow N(\mathcal{C})$  be the respective $\infty$-categorical limits of $\mathcal{F}_{\infty}$ and $\mathcal{G}_{\infty}$.  These respectively give rise to the collection of morphisms
$$\eta_{\mathcal{I}} = \{ lim(\mathcal{F}_{\infty})\rightarrow \mathcal{F}(i)\}_{i\in \mathcal{I}},$$
$$\gamma_{\mathcal{I} }= \{ lim(\mathcal{G}_{\infty})\rightarrow \mathcal{G}(i)\}_{i\in \mathcal{I}}.$$
\noindent
As discussed above, all this data gives rise to the following diagram
$$
\xymatrix{
Map_{\mathcal{C}}(A, lim(\mathcal{F}_{\infty})) \ar[r] & holim(\mathcal{F}_A)\ar[d] \\ Map_{\mathcal{C}}(A, lim(\mathcal{G}_{\infty})) \ar[r] &holim(\mathcal{G}_A)
}$$
All the arrows in this diagram are weak equivalences of simplicial sets.  Thus we get a canonical isomorphism $Map_{\mathcal{C}}(A, lim(\mathcal{F}_{\infty})) \cong Map_{\mathcal{C}}(A, lim(\mathcal{G}_{\infty}))$ in $\mathcal{H}$.  This induces a canonical bijection $Hom_{h(\mathcal{C})}(A, lim(\mathcal{F}_{\infty})) \cong Hom_{h(\mathcal{C})}(A, lim(\mathcal{G}_{\infty}))$.  This is natural in $A$ and thus induces a natural  isomorphism between their respective hom-functors in $h(\mathcal{C})$.  By the Yoneda lemma, this induces a canonical isomorphism $lim(\mathcal{F}_{\infty})\cong lim(\mathcal{G}_{\infty})$ in $h(\mathcal{C})$.  We let 
$$\varphi_{\infty} :lim(\mathcal{F}_{\infty})\rightarrow lim(\mathcal{G}_{\infty})$$
\noindent
be a lift of this isomorphism to a morphism in $\mathcal{C}$.  By construction $\varphi_{\infty}$ is an equivalence in $\mathcal{C}$ and is canonical up to homotopy.    
\end{proof}
\noindent
We remark that the proof of theorem 2 shows that if we take any natural transformation $\varphi:\mathcal{F}\rightarrow G$, then it induces a canonical (up to homotopy) morphism $\varphi_{\infty}$ on their respective limits, if they exist.

By \S 1.3.1 of \cite{GD},  the $\infty$-category $\textbf{dg-Cat}_{\infty}^{tri}$ admits (small) limits.  This fact is fundamental and justifies why the theory of triangulated dg-categories is superior to triangulated categories, where limits do not exists.

\section{$\mathcal{D}$-modules on Algebraic Stacks}
\noindent
For a detailed treatment of the material in this section we refer the reader to \S1, \S3 and \S6 of \cite{DH}.
\\ \\
\noindent
For the rest of this paper set $k=\C$. Let $X$ be a smooth complex algebraic variety.    Let $D$-$mod(X)$ denote the $\C$-linear abelian category of algebraic (left) $D$-modules on $X$ (\S1.2 \cite{DH}).  Let $D$-$mod_{rh}(X)$ denote the full subcategory of regular holonomic $D$-modules on $X$ (\S6.1 \cite{DH}).

The category $D$-$mod(X)$ is $\C$-linear, hence it follows from \S1.3.1 of \cite{LA}, that bounded complexes of $D$-modules on $X$ naturally form a dg-category, which we denote by $Ch^b(D$-$mod(X))$.  Let $D^b(X)\subset Ch^b(D$-$mod(X))$ denote the full sub-dg-category of bounded complexes of injective objects.  We call $D^b(X)$ the bounded derived dg-category of $D$-modules on $X$.  This is a triangulated dg-category and the classical bounded derived category of $D$-modules is canonically equivalent to its homotopy category.
We define the bounded derived dg-category of regular, holonomic $D$-modules to be $D^b_{rh}(X) \subset D^b(X)$, the full sub-dg-category with regular, holonomic homology.  We equip $D^b_{rh}(X)$ with its standard $t$-structure, with heart isomorphic to the category of regular holonomic $D$-modules on $X$.  
\\ \\
\noindent
Let $f:X\rightarrow Y$ be a morphism of smooth algebraic varieties over $\C$.  Following \S1.5 of \cite{DH}, we have the direct and inverse image (Ch($\C$)-enriched) functors:
$$f_* : D^b(X) \rightarrow D^b(Y) \;\;\; f^{\dagger}: D^b(Y)\rightarrow D^b(X).$$
We define the shifted inverse image functor to be
$$f^{!}  := f^{\dagger} [dimX - dimY].$$
All of these functors preserve the sub-dg-category of regular holonomic $D$-modules.
\\ \\
\noindent
There is also the Verdier involution
$$\mathbb{D}_X : D^b(X)^{op} \rightarrow D^b(X),$$
defined in the usual way, using the canonical bundle on X (\S2.6 \cite{DH}).  This involution descends to an involution of the derived dg-category of regular, holomonic $D$ modules.  
\\ \\
\noindent
We define the functors
$$f_! : D(X) \rightarrow  D(Y) \;\;\;\;\;\;\; f^{\bigstar} : D(Y)\rightarrow D(X)$$
$$ f_! = \mathbb{D}_Y \circ f_* \circ \mathbb{D}_X\;\; \;\;\;\;\;\; f^{\bigstar} = \mathbb{D}_X\circ f^{!}\circ\mathbb{D}_Y.$$
We have the following standard facts:

\begin{itemize}
\item These functors descend to functors on the sub dg-categories of regular holonomic $D$-modules.
\item The functors $(f_!, f^{!})$ and $( f^{\bigstar}, f_*)$ are adjoint pairs.
\item If $f$ is proper then $f_* = f_!$.
\item For $f$ smooth, the non-shifted inverse image functor $f^{\dagger}$, is $t$-exact.
\item If $f$ is smooth of relative dimension $d$ then there is a  canonical natural isomorphism $f^{\bigstar}\cong f^![-2d]$.
\item If $X, Y$ and $Z$ are smooth complex algebraic varieties and $f:X\rightarrow Y$ and $g:Y\rightarrow Z$ are two morphisms, then there is a canonical isomorphism of functors $(gf)^{!}\cong f^!g^!$.  The same holds for $\bigstar$ and $\dagger$.  
\end{itemize}
We now address the problem of defining suitable categories of $D$-modules on algebraic stacks.

Observe that $\dgt$ is a strict 2-category, where 2-morphisms are given by $\infty$-natural transformations.  If $\mathcal{I}$ is a small ordinary category and we are given $\eta: \mathcal{I}\rightarrow \mathbf{Sch}_{\C}$, an $\mathcal{I}$-diagram of smooth complex varieties $(X_i)_{i\in \mathcal{I}}$,  then the pullback functor $\bigstar$ induces the pseudofunctor:
$$\mathfrak{Drh}_{\eta}^{\bigstar}: \mathcal{I}^{op}\rightarrow \dgt.$$
$$\;\;\; \;i \mapsto D_{rh}^b(X_i).$$
Let $\widehat{\mathfrak{Drh}_{\eta}^{\bigstar}}$ denote the strictification of this pseudofunctor.  Recall that $\widehat{\mathfrak{Drh}_{\eta}^{\bigstar}}$ is a strict 2-functor $\mathcal{I}^{op}\rightarrow \dgt$ equipped with a canonical pseudonatural equivalence $\mathfrak{Drh}_{\eta}^{\bigstar}\rightarrow\widehat{\mathfrak{Drh}_{\eta}^{\bigstar}}$, satisfying the usual universal property.  Thus we may consider $\widehat{\mathfrak{Drh}_{\eta}^{\bigstar}}$ as an ordinary functor of 1-categories and apply the results of the previous paragraph.

Throughout the rest of this paper, let $\xs$ be a smooth complex algebraic (Artin) stack which admits an algebraic variety as a smooth atlas. Let $\pi: X\rightarrow \xs$ be such an atlas.  We may associate to this data the Cech smooth simplicial scheme  $X_{\bullet}\rightarrow \xs$.  More precisely, this is the simplicial scheme 
$$\pi_{\bullet}:  \Delta^{op}\rightarrow \textbf{Sch}_k$$
$$\;\;\; \;[n] \mapsto X_n,$$
where $X_n$ is the $(n+1)$-fiber product of $X$ with itself over $\xs$.  As above, this induces the strict 2-functor:
$$\widehat{\mathfrak{Drh}^{\bigstar}_{\pi_{\bullet}}}: \Delta \rightarrow \dgt$$
This induces the simplicially enriched functor 
$$\widehat{\mathfrak{Drh}^{\bigstar}_{\pi_{\bullet}}}_{\Delta}: \Delta \rightarrow \dgtd,$$
where we consider $\Delta$ as a simplicial category. By taking the nerve we have the $\infty$-functor:
$$\widehat{\mathfrak{Drh}^{\bigstar}_{\pi_{\bullet}}}_{\infty}: N(\Delta) \rightarrow \dgti.$$

\begin{defin}
We define the bounded derived dg-category of regular, holonomic  $D$-modules on $\xs$ to be 
$$D_{rh}^b(\xs):=lim(\widehat{\mathfrak{Drh}^{\bigstar}_{\pi_{\bullet}}}_{\infty}).$$
\noindent
\end{defin}
\noindent
A standard argument shows this to be independent of the choice of atlas. By applying the $\infty$-categorical version of Grothendieck's pseudofunctor/fibred-category equivalence (\S 3.3.3.2 \cite{LT})  we have the following concrete description of objects of this limit:  An object $M \in D_{rh}^b(\xs)$ is an assignment for every non-negative integer $n$, an  object $M_{X_n} \in D_{rh}^b(X_n)$, and for every morphism $\phi:[n] \rightarrow [m]$ in $\Delta$ (inducing a morphism $f_{\phi}: X_m\rightarrow X_n $) an isomorphism $f^{\bigstar}_{\phi}(M_{X_m})\cong M_{X_n}$, where the collection of such morphisms forms a homotopy-coherent diagram.
\\ \\
\noindent
This concrete description allows us to put a $t$-structure on $D_{rh}^b(\xs)$.  Recall that for a smooth morphism $f$, of relative dimension $d$, we have canonical isomorphisms $f^{\dagger}\cong f^![-d]\cong f^{\bigstar}[d]$. Moreover $f^{\dagger}$ is $t$-exact.  Let $d_{\pi}$ be the relative dimension of our fixed atlas. Thus $D^b_{rh}(\xs)$ inherits a canonical $t$-structure given by the following: $M\in D^b_{rh}(\xs)_{\geq 0}$ if and only if $M_{X_0}[d_{\pi})]\in D_{rh}^b(X_0)_{\geq 0}$ and $M\in D^b_{rh}(\xs)_{\leq 0}$ if and only if $M_{X_0}[d_{\pi})]\in D_{rh}^b(X_0)_{\leq 0}$.  We define category of regular, holonomic $D$-modules on $\xs$ to be the heart of this $t$-structure.

\section{Constructible Sheaves on Algebraic Stacks}
\noindent
let $\xa$ be a smooth complex analytic space.  Let $\shx$ be the $\C$-linear Abelian category of sheaves (in the analytic topology) of $\C$-vector spaces on $\xa$.  We say that a sheaf $\mathcal{F}\in \shx$, is \textit{locally constant constructible}, abbreviated as \textit{llc}, if it is locally constant and has finite dimensional stalks. We say that $\mathcal{F}$ is constructible if it is \textit{llc} on each piece of some (analytic) stratification. We denote the full Abelian subcategory of constructible sheaves on $\xa$ by $\shacx\subset \shx$.  
\\ \\
\noindent
Now let $X$ be a smooth algebraic variety over $\C$ and let $X^{an} =\xa$ denote its complex analytification.  We say that a sheaf $\mathcal{F}\in \shacx$ is \textit{algebraically} constructible if it is \textit{llc} on each piece of some \textit{algebraic} stratification.  We denote this full subcategory by $\shcx \subset \shacx$.
\\ \\
\noindent
We now define the dg-category of (algebraically) constructible sheaves on $\xa$.  Let $\textbf{Ch}^b(\xa, \C)$ denote the dg-category of bounded complexes of objects in $\shx$.   Because $\shx$ is a Grothendieck Abelian category we define the bounded, derived dg-category of sheaves in complex vector spaces on $\xa$ to be the full sub-dg-category $\textbf{dg-Mod}^b(\xa, \C)\subset \textbf{Ch}^b(\xa, \C) $ given by complexes of injective objects. The homotopy category of $\textbf{dg-Mod}^b(\xa, \C)$ is just the usual bounded derived category of $\shx$.  

We define the dg-category of \textit{algebraically} constructible sheaves on $\xa$ to be the full sub-dg-category $\textbf{dg-Mod}^b_c(\xa, \C) \subset \textbf{dg-Mod}^b(\xa, \C)$, given by objects with algebraically constructible homology.  The dg-category $\textbf{dg-Mod}^b_c(\xa, \C)$ is triangulated and we equip it with the \textit{perverse} $t$-structure for the middle perversity (\S8.1 \cite{DH}).
\\ \\
\noindent
As in the case of $D$-modules, these dg-categories are subject to the six functor formalism.
Let $X$ and $Y$ be smooth algebraic varieties over $\C$ and  $f: X\rightarrow Y$ be a morphism.  This induces a morphism $f_{an}: X^{an}\rightarrow Y^{an}$.  As explained in \S4.5 of \cite{DH}, this induces the cohomological functors: $$f^{-1}: \textbf{dg-Mod}^b_c(Y^{an}, \C)\rightarrow \textbf{dg-Mod}^b_c(X^{an}, \C),$$ $$f_* :\textbf{dg-Mod}^b_c(X^{an}, \C)\rightarrow \textbf{dg-Mod}^b_c(Y^{an}, \C).$$ 
 $$f^{!}: \textbf{dg-Mod}^b_c(Y^{an}, \C)\rightarrow \textbf{dg-Mod}^b_c(X^{an}, \C),$$
  $$f_!: \textbf{dg-Mod}^b_c(X^{an}, \C)\rightarrow \textbf{dg-Mod}^b_c(Y^{an}, \C),$$
We again have the following standard facts:
\begin{itemize}
\item For $f$ as above we have the adjoint pairs $(f^{-1}, f_*)$ and $(f_!, f^!)$.
\item  For $f$ smooth of relative dimension $d$  there is a canonical isomorphism $f^{-1}\cong f^!  [-2d]$. Moreover the functor $f^{-1}[d]$ is $t$-exact (with respect to the perverse $t$-structure).
\end{itemize}
\noindent
Let $\xs$ be a smooth algebraic stack over $\C$, which admits a complex algebraic variety as a smooth atlas. Let $\pi: X\rightarrow \xs$ be such an atlas.   Let $\xs^{an}$ be the associated complex analytic stack.   By applying the complex analytification functor we get a smooth complex analytic  atlas $X^{an}\rightarrow \xs^{an}$.  From this we may form the Cech smooth simplicial cover as in the algebraic case, $X^{an}_{\bullet}\rightarrow \xs^{an}$.  We denote the associated simplicial complex analytic space by $\pi_{\bullet}^{an}$.  This gives rise to the pseudofunctor:
$$\mathfrak{Con}^{-1}_{\pi_{\bullet}^{an}}: \Delta \rightarrow \dgt$$
$$\;\;\;\;\;\;\;\;\;\;\;\;\;\;\;\;\;\;\;\;\;\;\;\;\;\;[n] \mapsto \textbf{dg-Mod}_c^b(X^{an}_n, \C).$$
Let $\widehat{\mathfrak{Con}^{-1}_{\pi_{\bullet}^{an}}}$ denote the associated strict 2-functor. As in the case of $D$-modules we associate to this data the corresponding $\infty$-functor:

$$\widehat{\mathfrak{Con}^{-1}_{\pi_{\bullet}^{an}}}_{\infty}: N(\Delta) \rightarrow \dgti.$$
\begin{defin}
We define the derived dg-category of algebraically constructible sheaves on $\xs$ to be the limit
$$\textbf{dg-Mod}_c^b(\xs^{an}, \C):=lim(\widehat{\mathfrak{Con}^{-1}_{\pi_{\bullet}^{an}}}_{\infty}).$$
\end{defin}
\noindent
Concretely, an object $M \in \textbf{dg-Mod}_c^b(\xs^{an}, \C)$ is an assignment for every non-negative integer $n$, an  object $M_{X_n} \in \textbf{dg-Mod}_c^b(X_n^{an}, \C)$, and for every morphism $\phi:[n] \rightarrow [m]$ in $\Delta$ (inducing a morphism $f_{\phi}: X_m\rightarrow X_n $) an isomorphism $f^{-1}_{\phi}(M_{X_m})\cong M_{X_n}$, where the collection of such isomorphisms forms a homotopy-coherent diagram.
 \\ \\
\noindent
Let $d_{\pi}$ be the relative dimension of our fixed atlas.  As in the case of $D$-modules, the triangulated dg-category $ \textbf{dg-Mod}_c^b(\xs^{an}, \C)$ inherits a (perverse) $t$-structure by decreeing that $M\in  \textbf{dg-Mod}_c^b(\xs^{an}, \C)_{\geq 0}$ if and only if $M_{X_0}[d_{\pi}]\in  \textbf{dg-Mod}_c^b(X_0^{an}, \C)_{\geq 0}$.  We define the Abelian category of perverse sheaves on $\xs$ to be the heart of this triangulated dg-category.

\section{The Riemann-Hilbert Correspondence}
Let $X$ be a smooth complex algebraic variety.  For $M \in D^b_{rh}(X)$ we denote the associated analytic $D$-module on the complex analtyic space $X^{an}$ by $M^{an}$.  Let $\Omega_{X^{an}}$ denote the canonical bundle on $X^{an}$.  Recall that $\Omega_{X^{an}}$ is equipped with a canonical (right) analytic $D$-module structure. As explained in \S4.2 of \cite{DH},  we define the de Rham functor:

$$\mathfrak{DR}_X: D_{rh}(X) \rightarrow \textbf{dg-Mod}_c^b(X^{an}, \C)$$
$$ \;\;\;\;\;\;\;\;\;\;\;\;\;\;\;\;\;\;\;\;\;\;\;M \mapsto \Omega_{X^{an}} \otimes^\mathbb{L}_{D_{X^{an}}} M^{an}.$$
\noindent
\begin{theoremrh}
For $X$ a smooth, complex algebraic variety the de Rham functor is a $t$-exact weak equivalence of triangulated dg-categories.
\end{theoremrh}

\begin{proof}
We remind the reader that we have fixed the standard $t$-structure on for $D$-modules and the (middle) perverse $t$-structure for constructible sheaves. A proof is given in \S 7.2.2 of \cite {DH}.
\end{proof}
\noindent
We now extend this result to algebraic stacks.  Let $\xs$ be a smooth complex algebraic stack which admits an algebraic variety as a smooth atlas.  Let $\pi: X\rightarrow \xs$ be such an atlas.

By \S7.1.1.1 of \cite{DH}, if $Y$ and $Z$ are smooth complex algebraic varieties and $f:  Y \rightarrow Z$ is a morphism then there is a canonical isomorphism of functors:
$$\mathfrak{DR}_Z\circ f^{\bigstar} \cong f^{-1}\circ \mathfrak{DR}_Y.$$
\noindent
This induces the pseudonatural transformation
$$\mathfrak{DR}_{\xs}:  \mathfrak{Drh}^{\bigstar}_{\pi_{\bullet}}\rightarrow \mathfrak{Con}^{-1}_{\pi^{an}_{\bullet}},$$
$$\;\;\;\;\;\;\;\;\;[n] \mapsto \mathfrak{DR}_{X_n}$$
\noindent
This in turn gives rise to the strict-natural transformation
$$\widehat{\mathfrak{DR}_{\xs}}:  \widehat{\mathfrak{Drh}^{\bigstar}_{\pi_{\bullet}}}\rightarrow \widehat{\mathfrak{Con}^{-1}_{\pi^{an}_{\bullet}}}$$
By the remark following Theorem 2, a canonical (up to homotopy) morphism
$$\widehat{\mathfrak{DR}_{\xs}}_{\infty}:D_{rh}^b(\xs)\rightarrow \textbf{dg-Mod}_c^b(\xs^{an}, \C)$$
in $\dgti$. We called $\widehat{\mathfrak{DR}_{\xs}}_{\infty}$ the $\infty$-categorical de Rham functor.
\begin{thmRH}
Let $\xs$ be a smooth complex algebraic stack, which admits an algebraic variety as a smooth atlas.  Then the $\infty$-categorical de Rham functor $\widehat{\mathfrak{DR}_{\xs}}_{\infty}$ is an equivalence in $\dgti$.  Moreover it induces a canonical equivalence between the category of regular, holonomic $D$-modules on $\xs$ and the category of perverse sheaves on $\xs$.
\end{thmRH}
\begin{proof}
The classical Riemann-Hilbert correspondence implies that the $\widehat{\mathfrak{DR}_{\xs}}$ becomes a natural isomorphism after composing each $\widehat{\mathfrak{Drh}^{\bigstar}_{\pi_{\bullet}}}$ and $\widehat{\mathfrak{Con}^{-1}_{\pi^{an}_{\bullet}}}$ with the canonical functor $\dgt\rightarrow h(\dgti)$.  Applying Theorem 2 we conclude that $\widehat{\mathfrak{DR}_{\xs}}_{\infty}$ is an equivalence in $\dgti$.  

The classical de Rham functor is $t$-exact for the standard $t$-structure on $D$-modules and the perverse $t$ structure on constructible sheaves.  It is clear by therefore that $\widehat{\mathfrak{DR}_{\xs}}_{\infty}$ induces a $t$-exact morphism in $\dgti$.  This concludes the proof.
\end{proof}

\bibliographystyle{plain}
\bibliography{bibliography}

\end{document}